\documentclass[11pt,letterpaper]{amsart}
\usepackage{amsmath,amsthm,amsfonts,amscd,amssymb,eucal,latexsym,mathrsfs,stmaryrd,enumitem,bm}
\usepackage{graphicx}
\usepackage{xcolor}
\graphicspath{ {./images/} }

\newtheorem{theorem}{Theorem}[section]
\newtheorem{corollary}[theorem]{Corollary}
\newtheorem{lemma}[theorem]{Lemma}

\theoremstyle{plain}

\theoremstyle{plain}
\newcounter{theoremintro}
\newtheorem{theoremi}[theoremintro]{Theorem}
\newtheorem{corollaryi}[theoremintro]{Corollary}

\newcommand{\cZ}{{\mathcal Z}}

\newcommand{\sB}{{\mathscr B}}
\newcommand{\sC}{{\mathscr C}}

\newcommand{\sF}{{\mathscr F}}

\newcommand{\sL}{{\mathscr L}}

\newcommand{\sN}{{\mathscr N}}

\newcommand{\sT}{{\mathscr T}}
\newcommand{\sU}{{\mathscr U}}

\newcommand{\Zb}{{\mathbb Z}}

\newcommand{\Nb}{{\mathbb N}}

\newcommand{\comp}{{\rm c}}

\newcommand{\eps}{\varepsilon}

\newcommand{\property}{\star}

\numberwithin{equation}{section}

\allowdisplaybreaks

\begin{document}

\title{Elementary amenability and almost finiteness}

\author{David Kerr}
\address{David Kerr,
Mathematisches Institut,
WWU M{\"u}nster, 
Einsteinstr.\ 62, 
48149 M{\"u}nster, Germany}
\email{kerrd@uni-muenster.de}

\author{Petr Naryshkin}
\address{Petr Naryshkin,
Institute for Advanced Study, 1 Einstein Dr, Princeton, NJ 08540, US}
\email{penaryshkin@ias.edu}

\date{October 30, 2025}

\begin{abstract}
We show that every free continuous action of a countably infinite elementary amenable group on a 
finite-dimensional compact metrizable space is almost finite. As a consequence,
the crossed products of minimal such actions are $\cZ$-stable and 
classified by their Elliott invariant.
\end{abstract}


\maketitle

\section{Introduction}

A basic principle in the study of group actions and their operator-algebraic crossed products
is that dynamical towers produce matricial structure,
and that dynamical towers with F{\o}lner (i.e., approximately invariant) shapes 
produce approximately central matricial structure. 

In the case of a free measure-preserving action
of a countable amenable group on a standard probability space, 
a theorem of Ornstein and Weiss gives the existence of
a family of disjoint F{\o}lner-shaped towers which cover all but a small piece 
of the space \cite{OrnWei87}. The smallness of this remainder, expressed in terms of the measure,
means that the multimatrix algebras generated by the towers 
partition the unit of the von Neumann algebra crossed product up to a small error in trace norm.
When combined with the approximate centrality that ensues from the F{\o}lnerness
of the tower shapes, this yields 
local approximation by finite-dimensional $^*$-subalgebras and hence hyperfiniteness
of the crossed product\footnote{This conclusion was originally derived in a more directly
operator-algebraic way by Connes as a consequence of his theorem on the equivalence
of injectivity and hyperfiniteness \cite{Con76}.},
and if the action is ergodic one obtains the unique hyperfinite II$_1$ factor.

The analogous C$^*$-theory 
involving actions of countable amenable groups on compact metrizable spaces is more 
complicated and, despite many significant advances, still incomplete. 
While a variety of different tools and techniques have been developed
over the last four decades \cite{Put90,EllEva93,LinPhi04,LinPhi10,StrWin11,TomWin13}, 
in large part stimulated by new ideas that have emerged from the
Elliott classification program, a basic pattern has crystallized through the course of recent 
progress, which we can summarize as follows.

A major point of departure from the measure-theoretic setting is the 
presence of dimensionality, which means that matrix models coming from the
dynamics will need to be continuous instead of discrete, or that they will need to 
be reconceptualized as order-zero maps from matrices into the algebra.
The first of these options, which connects to ideas of dimension-rank ratio and dimension growth,
is particularly powerful
when dimensional regularity hypotheses on the space (e.g., finite covering dimension)
are replaced by more general ones on the dynamics such as zero mean dimension
or the small boundary property, which have so far only yielded to
this approach. 
In this case the goal has been to show that the C$^*$-algebra crossed
product is $\cZ$-stable, 
which has been achieved in the case of 
free minimal $\Zb^d$-actions with zero mean dimension \cite{EllNiu17,Niu19a}, although
it seems to be a difficult problem to establish similarly general statements
for other acting groups (see however \cite{Niu19b}).
 
The second option can be developed in two different ways. In general,
one is confronted
with the problem that, unlike for the trace norm approximations in the von Neumann algebraic
setting, the operator norm approximations that are essential for unraveling C$^*$-structure
cannot be done by purely spatial means and instead must be implemented with the help
of spectral constructions, even if the space is zero-dimensional. 
This is true both for the approximation of the unit
in the tower configurations and for the approximate centrality
demanded of the matrix models. One possibility is to drop the F{\o}lner requirement and
allow enough overlap between towers so that the bump
functions implementing approximate centrality will form a genuine partition of unity.
Control on the multiplicity of this overlapping 
will lead to estimates on the nuclear dimension 
of the crossed product \cite{WinZac10,HirWinZac15,Sza15,SzaWuZac19}, 
and can be formalized at the dynamical level through the notions of 
dynamic asymptotic dimension \cite{GueWilYu17} and tower dimension \cite{Ker20}.
Another idea, formalized in the definition of almost finiteness 
(see Section~\ref{S-preliminary}), is to insist on the 
F{\o}lnerness and disjointness of 
the towers and express the shortfall in the partitioning by means of a topological 
version of Cuntz subequivalence, which is then sufficient to imply that the
crossed product is $\cZ$-stable \cite{Ker20}. 
The two approaches are connected at the dynamical level
by the observation that a free action
on a space of finite covering dimension is almost finite if its
dynamic asymptotic dimension or tower dimension is finite
(see Corollary~6.2 of \cite{KerSza20} and Theorem~5.14 of \cite{Ker20}).
The first approach has proven to be very effective for certain classes of groups,
as in the recent paper \cite{ConJacMarSewTuc20} where
finite dynamic asymptotic dimension (and hence also almost finiteness) is established
for free actions of many solvable groups, including polycyclic groups and  
the lamplighter group, on zero-dimensional compact metrizable spaces.
On the other hand, such use of the dimensional idea of controlled overlapping
from which a proof of finite nuclear dimension can be derived
has invariably required the space to have finite covering dimension and the group to have
finite asymptotic dimension, the latter being a property that excludes many amenable groups.
From this perspective almost finiteness has turned out to be more broadly applicable,
and indeed has been shown to hold for free minimal actions of groups 
with local subexponential growth on zero-dimensional compact 
metrizable spaces \cite{DowZha17,DowZha19},
as well as for a generic free minimal action of any countably infinite
amenable group on the Cantor set \cite{ConJacKerMarSewTuc18}.
Moreover, by Theorem~7.6 of \cite{KerSza20},
for a fixed countably infinite group $G$, 
if every free action of $G$ on a zero-dimensional compact metrizable space is almost finite,
then every free action of $G$ on a finite-dimensional compact metrizable space 
(and in fact every free action with the topological small boundary property) 
is almost finite.

Given this dynamical picture it comes as quite a surprise that,
for unital simple nonelementary separable C$^*$-algebras,
finite nuclear dimension is actually equivalent to $\cZ$-stability (the forward implication
was proven in \cite{Win12}, while the backward implication was recently 
established in \cite{CasEviTikWhiWin18} 
after a string of partial results beginning with the breakthrough in \cite{MatSat14}).
The significance of these two regularity properties and their equivalence in this context is that
the class of unital simple separable C$^*$-algebras having finite nuclear dimension
and satisfying the UCT is 
classified by the Elliott invariant (ordered $K$-theory paired with tracial states) \cite{Kir94,Phi00,GonLinNiu15,EllGonLinNiu15,TikWhiWin17},
and every stably finite member of this class is an inductive limit of subhomogenous C$^*$-algebras
whose spectra have covering dimension at most two \cite{Ell96} 
(see Theorem~6.2(iii) in \cite{TikWhiWin17}).
The crossed product of a free minimal action of a countable amenable group 
on a compact metrizable space will therefore be covered by these classification 
and structure results as soon as it is known to have finite nuclear
dimension, or equivalently be $\cZ$-stable
(note that the UCT is automatic by \cite{Tu99}, and that amenability
implies the existence of an invariant Borel probability measure and hence of a tracial
state, which ensures stable finiteness in view of simplicity).
That $\cZ$-stability does not always hold in this context, even
for free minimal $\Zb$-actions, was shown in \cite{GioKer10}.

In this paper we establish the following theorem, which generalizes the almost finiteness
result from \cite{ConJacMarSewTuc20}.
By definition, the class of elementary amenable groups is the smallest class
of groups which contains all finite groups and Abelian groups and is closed under 
taking subgroups, quotients, extensions, and direct limits.
This class includes all solvable groups, is closed under taking wreath products,
and contains many groups with both exponential growth and infinite asymptotic dimension, such as $\Zb\wr\Zb$. 
A finitely generated infinite amenable group cannot be 
elementary amenable if it has intermediate growth \cite{Cho80},
like the prototypical Grigorchuk group \cite{Gri84}, 
or if it is simple, like the commutator subgroup of the topological full group 
of a minimal subshift \cite{Mat06,JusMon13}.

\begin{theoremi}\label{T-A}
Every free continuous action of a countably infinite elementary amenable group 
on a finite-dimensional compact metrizable space is almost finite.
\end{theoremi}

By Theorem~7.6 of \cite{KerSza20}, as mentioned three paragraphs above,
it is enough to prove the theorem in the case
of zero-dimensional compact metrizable spaces, which is what we will do,
also without the assumption that the countable group be infinite
(for finite groups, an action as in Theorem~\ref{T-A} 
is almost finite if and only if the
space is zero-dimensional, and so Theorem~\ref{T-A}
is actually false in this case, and Theorem~7.6 of \cite{KerSza20}
is only valid when the group is infinite).
In other words, we will establish that every countable elementary amenable group $G$
satisfies the following property:
\begin{itemize}
\item[($\property$)] every free continuous action
of $G$ on a zero-dimensional compact metrizable space is almost finite.
\end{itemize}
As is clear from the quantification over finite subsets in the definition of almost finiteness,
property ($\property$) is preserved under taking countable direct limits.
In Theorem~\ref{T-finite} we prove that property ($\property$) is preserved
under finite extensions, while in Theorem~\ref{T-integers}, to which most of our
efforts will be devoted, we show that
property ($\property$) is preserved under extensions by $\Zb$. 
Actually none of these three permanence properties require the zero-dimensionality
hypothesis on the space, but in order to bootstrap our way to the final
result we will rely on the fact 
that property ($\property$) holds for the trivial group,
as can be seen from the definition of almost finiteness
(see Section~\ref{S-preliminary}) by taking the tower bases therein
to form a fine enough clopen partition of the space 
and the proportionally small subsets of the tower shapes to be empty.
To conclude that property ($\property$) holds for all countable
elementary amenable groups we can then
appeal to a theorem of Osin \cite{Osi02} which, refining a result of Chou \cite{Cho80},
characterizes this class as the smallest class of groups that contains the trivial group
and is closed under taking countable direct limits and extensions by $\Zb$ and finite groups.
Note that, in view of \cite{DowZha17,DowZha19}, 
we actually obtain property ($\property$) and hence also Theorem~\ref{T-A}
for a broader class of groups, namely the smallest class that contains all countable groups of 
local subexponential growth and is closed under taking countable direct limits and 
extensions by $\Zb$ and finite groups.

One of the novelties of the proof of Theorem~\ref{T-integers} is that it integrates
conceptual aspects from all three of the approaches that we sketched above (corresponding
to the regularity properties of zero mean dimension, finite dynamic asymptotic dimension, 
and almost finiteness). The argument proceeds by applying a recursive disjointification
procedure to an initial collection of overlapping open towers whose levels have
boundaries of upper $H$-density zero.
The shapes of these towers are 
F{\o}lner rectangles in the semidirect product $H\rtimes\Zb$, and the towers
generated by the restrictions of these shapes to $H$ cover all but a piece of 
the space with small upper $H$-density, as can be arranged using the 
hypothesized almost finiteness of the $H$-action. 
When $H$ is infinite these rectangles are 
chosen to be thin in the $\Zb$ direction and tall (i.e, much larger) in the $H$ direction,
in which case the multiplicity of the overlapping of the towers is small in proportion to the 
size of their shape in the $H$ direction,
very much in the spirit of the small dimension-rank ratios that appear in the
proof of $\cZ$-stability from zero mean dimension in \cite{EllNiu17}
and in the general study of inductive limits in classification theory.
This allows us to generate, by a recursive Ornstein--Weiss-type disjointification process as in \cite{ConJacKerMarSewTuc18,KerSza20},
a collection of open towers whose shapes will be F{\o}lner as long as they are proportionally
not too small within the ambient shapes of the initial towers from which they are created.
The union of the towers whose shapes are not sufficiently F{\o}lner can be divided into
two subsets, one which is small in upper $H$-density and the other whose points can be donated
to the F{\o}lner towers without significantly affecting the approximate invariance of the tower shapes.
The set with small upper $H$-density can be absorbed, 
via comparison, using the almost finiteness of the $H$-action.
In fact the initial towers will themselves need to be shaved down a little bit
at the outset in order to achieve the F{\o}lner disjointification 
(via the tiling argument captured in Lemma~\ref{L-intersections}), 
but this loss will also be small 
in upper $H$-density and can likewise be absorbed.

Unlike in an Ornstein--Weiss tiling, we do not need to repeat the disjointification process across 
different scales, as the geometry of our situation ensures that the tower shapes will already be F{\o}lner 
if they are proportionally greater than some small $\eps$ in the ambient initial tower. In the Ornstein--Weiss setting
this F{\o}lnerness can only be guaranteed if the proportion is greater than $1-\eps$ for some small $\eps$, 
a situation which occasions the additional recursion over different scales in order to 
geometrically increment the amount of coverage.

By our discussion prior to the statement of Theorem~A, 
we obtain the following corollary. 
The precise link to $\cZ$-stability is given by
Theorem~12.4 of \cite{Ker20}, which asserts that,
given an almost finite free minimal action $G\curvearrowright X$ of a countably infinite 
group on a compact metrizable space, the crossed product 
$C(X)\rtimes G$ is $\cZ$-stable (note that almost finiteness implies that $G$ is 
amenable and so the reduced and full crossed products coincide in this case). 

\begin{corollaryi}
The crossed products of free minimal actions of countably infinite elementary amenable groups on
finite-dimensional compact metrizable spaces are classified by their Elliott invariant 
and are simple inductive limits of subhomogeneous C$^*$-algebras
whose spectra have covering dimension at most two.
\end{corollaryi}

Theorem~\ref{T-A} also has consequences for topological full groups and homology. 
Let $G\curvearrowright X$ be free continuous action of a countably infinite 
elementary amenable group on the Cantor set. Denote by
$[[G\curvearrowright X]]$ the topological full group of the action
and by $[[G\curvearrowright X]]_0$ the subgroup of $[[G\curvearrowright X]]$ 
generated by the elements of finite order whose powers have clopen fixed point sets.
In Section~7 of \cite{Mat12} Matui defines an index map $I$ from 
$[[G\curvearrowright X]]$ to the first homology group $H_1 (G\curvearrowright X)$
with integer coefficients. 
The fact that the action $G\curvearrowright X$
is almost finite implies, by Corollary~7.16 of \cite{Mat12}, that $I$ is surjective
and has kernel $[[G\curvearrowright X]]_0$, so that it induces an isomorphism   
$H_1 (G\curvearrowright X) \cong [[G\curvearrowright X]]/[[G\curvearrowright X]]_0$.
If the action is in addition minimal then the commutator subgroup of 
$[[G\curvearrowright X]]$ is simple (by Theorem~4.7 of \cite{Mat15})
and equal to the alternating group 
${\sf A} (G\curvearrowright X)$ (by Theorem~4.7 of \cite{Mat15}
and Theorem~4.1 of \cite{Nek19}).

While this paper was being revised for publication, the second author discovered a simple
proof of Theorem~\ref{T-integers} which also works more generally when the extension of $H$ by $\Zb$ 
is replaced by any countable extension $G$ of $H$ \cite{Nar24},
thereby delivering a considerably larger class of amenable groups for which the conclusion of Theorem~\ref{T-A} holds.
Instead of verifying almost finiteness directly as we do here, the argument in \cite{Nar24} 
proceeds by establishing comparison for the action $G\curvearrowright X$ through the use of the action of $G$ 
on the space of $H$-invariant Borel probability measures together
with the castles and multiset comparison that the almost finiteness of the $H$-action furnishes.

The main body of the paper begins in Section~\ref{S-preliminary} with 
some general terminology and notation. The case of finite extensions
(Theorem~\ref{T-finite})
is treated in Section~\ref{S-finite}, while Sections~\ref{S-lemmas}
and \ref{S-integers} are devoted to extensions by $\Zb$ 
(Theorem~\ref{T-integers}). Section~\ref{S-lemmas} contains two technical 
lemmas for use in the proof of Theorem~\ref{T-integers}, which occupies the bulk
of Section~\ref{S-integers}.
\medskip

\noindent{\it Acknowledgements.}
The first author was partially supported by NSF grant DMS-1800633.
Funding was also provided by the Deutsche Forschungsgemeinschaft 
(DFG, German Research Foundation) under Germany's Excellence Strategy EXC 2044-390685587, 
Mathematics M{\"u}nster: Dynamics--Geometry--Structure, and by the SFB 1442 of the DFG.
Both authors were affiliated with Texas A\&M University during the initial stages of this work. 
We thank Xin Ma and Brandon Seward for comments and corrections.

\section{General terminology and notation}\label{S-preliminary}

We write $e$ for the identity element of a group.

For finite sets $K$ and $F$ of a group $G$, we define the $K$-boundary of $F$ by
\[
\partial_K F = \{ t\in G : Kt \cap F\neq\emptyset \text{ and } Kt \cap (G\setminus F)\neq\emptyset \} .
\]
For $\delta > 0$, we say that $F$ is {\it (left) $(K,\delta )$-invariant} if
$|\partial_K F| \leq\delta |F|$. By the F{\o}lner characterization of amenability,
the group $G$ is amenable if and only if it admits a nonempty finite $(K,\delta )$-invariant set
for every finite set $K\subseteq G$ and $\delta > 0$.

Let $G\curvearrowright X$ be a free continuous action 
of a countable group on a compact metric space
(we only consider free actions in this paper).
By a {\it tower} we mean a pair $(S,V)$ where $S$ is a finite subset of $G$ (the {\it shape})
and $V$ is a subset of $X$ (the {\it base})
such that the sets $sV$ for $s\in S$ (the {\it levels}) are pairwise disjoint.
The tower is {\it open} if $V$ is open and {\it clopen} if $V$ is clopen.
The {\it footprint} of the tower is the set $SV$.

A {\it castle} is a finite collection $\{ (S_i ,V_i ) \}_{i\in I}$ of towers such that
the sets $S_i V_i$ for $i\in I$ are pairwise disjoint. The castle is {\it open} if each 
of the towers is open and {\it clopen} if each of the towers is clopen.
The {\it footprint} of the castle is the set $\bigsqcup_{i\in I} S_i V_i$.

Let $A$ and $B$ be subsets of $X$. We say that $A$ is {\it subequivalent} to $B$
and write $A\prec B$ if for every closed set $C \subseteq A$ there are 
finitely many open sets $U_1 , \dots , U_n$ which cover $C$
and elements $s_1 , \dots , s_n \in G$ such that the sets $s_i U_i$ for 
$i=1,\dots , n$ are pairwise disjoint and contained in $B$. 
For a nonnegative integer $m$ we write $A\prec_m B$ 
if for every closed set $C\subseteq A$ there are a finite collection $\sU$ 
of open subsets of $X$ which cover $C$, an $s_U \in G$ for every $U \in \sU$, 
and a partition of $\sU$ into subcollections $\sU_0 , \dots , \sU_m$ 
such that for every $i = 0,...,m$ the sets $s_U U$ for $U \in \sU_i$ 
are pairwise disjoint and contained in $B$. 

The action $G\curvearrowright X$ has {\it comparison} if 
$A\prec B$ for all nonempty open sets $A,B\subseteq X$
which satisfy $\mu (A) < \mu (B)$ for every $G$-invariant 
Borel probability measure $\mu$ on $X$.
It has {\it $m$-comparison} for a nonnegative integer $m$ 
if $A \prec_m B$ for all nonempty open sets $A,B \subseteq X$ which satisfy 
$\mu (A) < \mu (B)$ for all $G$-invariant Borel probability measures $\mu$ on $X$.
In these definitions one can equivalently take $A$ to range over closed sets
instead of open ones (Proposition~3.4 of \cite{Ker20}).

The action $G\curvearrowright X$ is {\it almost finite} if for every $n\in\Nb$,
finite set $K\subseteq G$, and $\delta > 0$ there exist 
\begin{enumerate}
\item an open castle $\{ (S_i , V_i ) \}$ 
each of whose shapes is $(K,\delta )$-invariant and each of whose
levels has diameter less than $\delta$, and

\item sets $S_i' \subseteq S_i$ such that $|S_i' | \leq |S_i |/n$ and
$X\setminus \bigsqcup_{i\in I} S_i V_i \prec \bigsqcup_{i\in I} S_i' V_i$.
\end{enumerate}
Note that the shape condition in (i) implies that $G$ is amenable.
In the case that $G$ is finite, the action $G\curvearrowright X$
is almost finite if and only if $X$ is zero-dimensional
(notice that for $n > |G|$ the sets $S_i'$ in the above definition
will have to be empty).

It is worth noting (although we will not need this fact) that
when $X$ is zero-dimensional we can characterize
almost finiteness by the existence, for every finite set $K\subseteq G$ and $\delta > 0$,
of an open castle whose shapes are $(K,\delta )$-invariant and whose footprint is 
the entire space $X$ (Theorem~10.2 of \cite{Ker20}).

When $G$ is amenable, the {\it upper} and {\it lower densities} 
(or {\it $G$-densities} if we wish to emphasize the 
acting group) of a set $A\subseteq X$ are defined by
\begin{gather*}
\overline{D}_G (A) = \inf_F \sup_{x\in X} \frac{1}{|F|} \sum_{s\in F} 1_A (sx)
\hspace*{3mm}\text{and}\hspace*{3mm}
\underline{D}_G (A) = \sup_F \inf_{x\in X} \frac{1}{|F|} \sum_{s\in F} 1_A (sx)
\end{gather*}
where $F$ ranges in both cases over the nonempty finite subsets of $G$.
Writing $M_G (X)$ for the set of all $G$-invariant Borel probability measures on $X$,
we can alternatively express the upper density as $\sup_{\mu\in M_G (X)} \mu (A)$ 
when $A$ is closed and the lower density as $\inf_{\mu\in M_G (X)} \mu (A)$ 
when $A$ is open (see Proposition~3.3 of \cite{KerSza20}).

{\it Almost finiteness in measure} for the action $G\curvearrowright X$
is defined in the same way as almost finiteness
except that condition (ii) is replaced by the requirement that 
$X\setminus \bigsqcup_{i\in I} S_i V_i$ have upper density less than $\delta$
(uniform smallness in measure). By Theorem~5.6 of \cite{KerSza20},
the action $G\curvearrowright X$
is almost finite in measure if and only if it has the
{\it small boundary property}, which asks that 
$X$ have a basis of open sets whose boundaries are null for every 
$G$-invariant Borel probability measure on $X$.
By Theorem~6.1 of \cite{KerSza20} the following are equivalent:
\begin{enumerate}
\item the action is almost finite,

\item the action is almost finite in measure and has comparison,

\item the action is almost finite in measure and has $m$-comparison for some $m$. 
\end{enumerate}

\section{Finite extensions}\label{S-finite}

Let $G$ be a finite extension of a countable group $H$. Let $G\curvearrowright X$
be a free continuous action on compact metrizable space.

\begin{theorem}\label{T-finite}
Suppose that the restricted action $H\curvearrowright X$ is almost finite. 
Then the action $G\curvearrowright X$ is almost finite.
\end{theorem}

\begin{proof}
Since the action $H\curvearrowright X$ is almost finite, 
by the results recalled in Section~\ref{S-preliminary} it has the small boundary property. 
Since every $G$-invariant Borel probability measure on $X$ is also $H$-invariant,
it follows that the action $G\curvearrowright X$ 
has the small boundary property. Thus to show that $G\curvearrowright X$
is almost finite it suffices, again by the discussion in Section~\ref{S-preliminary}, 
to prove that it has $m$-comparison for some $m$. Suppose that for some open sets 
$A, B \subseteq X$ we have $\mu(A) < \mu(B)$ for every $G$-invariant 
Borel probability measure $\mu$ on $X$. 
Let $g_1, \dots, g_n$ be representatives for the left cosets of $H$ in $G$ with $g_1 = e$. 
Since the action $H\curvearrowright X$ is almost finite, it follows from the proof of 
Lemma~7.4 in \cite{KerSza20} that the set $A$ can be covered by $n+1$ open sets 
$A_1 , \dots , A_{n+1}$ such that $\nu(A_i) < \frac{1}{n}\nu(A)$ for every $i$ and 
every $H$-invariant Borel probability measure $\nu$ on $X$ 
(to construct the $(n+1)$st set take the 
closed complement of the footprint of the open castle in the proof of Lemma~7.4 in \cite{KerSza20}
and enlarge it to an open set whose measure is only slightly larger
for every $H$-invariant Borel probability measure,
as is possible by Lemma~3.3 
in \cite{Ker20}). Given such a measure $\nu$, the Borel probability measure 
\[
\overline{\nu}(D) = \frac{1}{n}(\nu(g_1D) + \nu(g_2D) + \ldots + \nu(g_nD))
\]
is $G$-invariant, and for every $i = 1,\dots ,n+1$ we have
\begin{align*}
\nu(A_i) \le n\overline{\nu}(A_i) < \overline{\nu}(A) 
< \overline{\nu}(B) &=\frac{1}{n}(\nu(g_1B) + \ldots + \nu(g_nB)) \\
&\le \nu(g_1B \cup \ldots \cup g_nB).
\end{align*}
Given a closed subset $C$ of $A$ and taking
closed sets $C_i \subseteq A_i$ such that $C = \bigcup_{i=1}^{n+1} C_i$,
the fact that the action of $H$ has comparison (by virtue of being almost finite)
thus yields, for every $i$,
pairwise disjoint open sets $U_{i, 1}, \dots, U_{i, k_i}\subseteq g_1B \cup \dots \cup g_nB$ 
and $h_{i,1} , \dots , h_{i,k_i} \in H$ such that $C_i \subseteq\bigcup_{k=i}^{k_i} h_{i,k} U_{i,k}$. 
For every $1 \le i \le n+1$ and $1 \le j \le n$ the sets $W_{i,j,k} := g_j^{-1}(g_jB \cap U_{i, k})$
for $k=1,\dots ,k_i$ are pairwise disjoint and contained in $B$,
and we have 
$C \subseteq \bigcup_{i=1}^{n+1} \bigcup_{j=1}^n \bigcup_{k=1}^{k_i} h_{i,k} g_j W_{i,j,k}$. 
This shows that $G\curvearrowright X$ has $n(n+1)$-comparison.
\end{proof}

\section{Two lemmas}\label{S-lemmas}

We collect here two lemmas that will be needed for the proof of Theorem~\ref{T-integers}.
The first concerns conditions under which the join of finitely many disjoint 
collections of subsets of a group, when restricted to an ambient F{\o}lner set $S$,
will mostly consist of F{\o}lner sets, where the degree of approximate invariance is prescribed
but necessarily much lower than that of $S$,
and ``mostly'' is understood in the sense that the exceptional sets 
will have collective size proportionally small relative to $|S|$.
The second lemma is a version of the implication 
(ii)$\Rightarrow$(i) in Theorem~6.1 of \cite{KerSza20} in which the hypotheses are 
relativized to a subgroup.

Let $\sF$ be a collection of subsets of $G$. We say that a set $A\subseteq G$
is {\it $\sF$-tileable} if there is a $T\subseteq G$ and sets $F_t \in \sF$ for $t\in T$ 
such that the sets $F_t t$ for $t\in T$ form a partition of $A$.

\begin{lemma}\label{L-intersections}
Let $n\in\Nb$. Let $K$ be a finite subset of $G$ and $\delta > 0$.
Let\ $\sF$ be a finite collection of $(K,\delta^3 /(8|K|^2 n))$-invariant finite subsets of $G$, and writing
$D = (\bigcup\sF )(\bigcup\sF )^{-1} $ let $S$ be a $(D^2 ,\delta^2 /(4|K|))$-invariant finite subset of $G$.
For each $i=1,\dots ,n$ let $\{ B_{i,1} , \dots , B_{i,m_i} \}$ be a finite disjoint collection 
of $\sF$-tileable finite subsets of $G$. For every $I\subseteq \{ 1,\dots ,n \}$
set $\Omega_I = \prod_{i\in I} \{ 1,\dots , m_i \}$ and for every $\omega\in\Omega_I$ set 
\[
B_\omega = \bigg( S \cap \bigcap_{i\in I} B_{i,\omega_i } \bigg)
\setminus \bigg( \bigcup_{i\in I^\comp} \bigsqcup_{j=1}^{m_i} B_{i,j} \bigg).
\]
Then the set $\Omega_0$ of all  
$\omega \in\Omega := \bigsqcup_{I\subseteq \{ 1,\dots ,n\}} \Omega_I$
such that $B_\omega$ fails to be $(K,\delta )$-invariant satisfies
$|\bigcup_{\omega\in\Omega_0} B_\omega | \leq \delta |S|$.
\end{lemma}

\begin{proof}
Note that the sets $B_\omega$ are pairwise disjoint.
Set $S_0 = S\cap \partial_{D^2} S$. Then $|S_0| \leq| \partial_{D^2} S| \leq \delta^2 (4|K|)^{-1} |S|$.
By $\sF$-tileability, for every $i=1,\dots ,n$ and $j=1,\dots ,m_i$ there are
a $T_{i,j} \subseteq G$ and $F_{i,j,t} \in\sF$ for $t\in T_{i,j}$ such that 
$B_{i,j} = \bigsqcup_{t\in T_{i,j}} F_{i,j,t} t$.
Write $T_{i,j}'$ for the set of all $t\in T_{i,j}$ such that $F_{i,j,t} t \subseteq S$,
and $T_{i,j}''$ for the set of all $t\in T_{i,j}$ such that $F_{i,j,t} t \subseteq S\setminus \partial_D S$.
Observe that, since $FF^{-1} (S\setminus \partial_D S)\subseteq S$ for every $F\in\sF$,
if $F_{i,j,t} t\cap (S\setminus \partial_D S)$ is nonempty for some
$t\in T_{i,j}$ then taking any element $s$ in this intersection we have 
$F_{i,j,t} t \subseteq F_{i,j,t} F_{i,j,t}^{-1} s \subseteq S$ and hence $t\in T_{i,j}'$.

For every $I\subseteq \{ 1,\dots ,n\}$ set 
$\Gamma_I = \prod_{i\in I} \{ (j,t) : 1\leq j\leq m_i ,\, t\in T_{i,j}'' \}$.
For each $\gamma = ((j_i , t_i ))_{i\in I} \in\Gamma_I$ define
\[
E_\gamma = \bigg(\bigcap_{i\in I} F_{i,j_i ,t_i} t_i \bigg) 
\setminus \bigg(\bigcup_{i\in I^\comp} \bigsqcup_{j=1}^{m_i} B_{i,j} \bigg) \subseteq S\setminus \partial_D S 
\]
and note that these sets over all 
$\gamma\in \Gamma := \bigsqcup_{I\subseteq\{ 1,\dots ,n\}} \Gamma_I$ are pairwise disjoint.
By the observation at the end of the first paragraph, for every $\gamma\in \Gamma$ we have
\[
E_\gamma = \bigg(\bigcap_{i\in I} F_{i,j_i ,t_i} t_i \bigg) 
\cap \bigg(\bigcup_{i\in I^\comp} \bigsqcup_{j=1}^{m_i} \bigsqcup_{t\in T_{i,j}'} (G\setminus F_{i,j,t} t)\bigg) 
\] 
and therefore
\begin{gather}\label{E-boundary}
\partial_K E_\gamma \subseteq \bigcup_{i=1}^n \bigsqcup_{j=1}^{m_i} \bigsqcup_{t\in T_{i,j}'} \partial_K F_{i,j,t} t .
\end{gather}

Write $\Gamma_0$ for the set of all $\gamma\in\Gamma$ such that 
$E_\gamma$ fails to be $(K,\delta /2)$-invariant. 
Since the sets $E_\gamma$ are pairwise disjoint, each element of $G$ belongs to 
$\partial_K E_\gamma$ for at most $|K|$ many $\gamma$, and so
\begin{gather}\label{E-boundary 2}
\sum_{\gamma\in\Gamma} |\partial_K E_\gamma |
\leq |K| \bigg|\bigcup_{\gamma\in\Gamma} \partial_K E_\gamma \bigg| .
\end{gather}
Also, since the sets $F_{i,j,t} t$ for $t\in T_{i,j}'$ are subsets of $S$ and each element of $S$ is contained in 
at most $n$ of them, we have 
\begin{gather}\label{E-multiplicity}
\sum_{i=1}^n \sum_{j=1}^{m_i} \sum_{t\in T_{i,j}'} |F_{i,j,t} t | \leq n|S| .
\end{gather}
We therefore obtain
\begin{align*}
\sum_{\gamma\in\Gamma_0} |E_\gamma |
&< \frac{2}{\delta} \sum_{\gamma\in\Gamma} |\partial_K E_\gamma | \\
&\stackrel{(\ref{E-boundary 2})}{\leq} \frac{2|K|}{\delta} \bigg| \bigcup_{\gamma\in\Gamma} \partial_K E_\gamma \bigg| \\
&\stackrel{(\ref{E-boundary})}{\leq} \frac{2|K|}{\delta} 
\bigg| \bigcup_{i=1}^n \bigsqcup_{j=1}^{m_i} \bigsqcup_{t\in T_{i,j}'} \partial_K F_{i,j,t} t \bigg| \\
&\leq \frac{2|K|}{\delta} 
\sum_{i=1}^n \sum_{j=1}^{m_i} \sum_{t\in T_{i,j}'} |\partial_K F_{i,j,t} t | \\
&\leq \frac{2|K|}{\delta} \cdot \frac{\delta^3}{8|K|^2 n} 
\sum_{i=1}^n \sum_{j=1}^{m_i} \sum_{t\in T_{i,j}'} |F_{i,j,t} t | \\
&\stackrel{(\ref{E-multiplicity})}{\leq} \frac{\delta^2}{4|K|} |S| .
\end{align*}

Set $S_1 = S_0 \cup \bigsqcup_{\gamma\in\Gamma_0} E_\gamma$.
Then 
\[
|S_1 | \leq |S_0 | + \sum_{\gamma\in\Gamma_0} |E_\gamma |
\leq \frac{\delta^2}{4|K|} |S| + \frac{\delta^2}{4|K|} |S| 
= \frac{\delta^2}{2|K|} |S| .
\]
Writing $\Omega_1$ for the set of all $\omega\in\Omega$ such that 
$|B_\omega \cap S_1 | > \delta (2|K|)^{-1} |B_\omega |$, we thereby obtain
\begin{gather*}
\bigg|\bigsqcup_{\omega\in\Omega_1} B_\omega \bigg| 
= \sum_{\omega\in\Omega_1} |B_\omega |
< \frac{2|K|}{\delta} \sum_{\omega\in\Omega_1} |B_\omega \cap S_1 |
\leq \frac{2|K|}{\delta} |S_1 |
\leq \delta |S| .
\end{gather*}
To complete the proof it is therefore enough to verify that $\Omega_0 \subseteq\Omega_1$.

Let $\omega\in\Omega\setminus\Omega_1$. 
Then there is a $\Gamma_\omega \subseteq \Gamma$ such that 
$(S\setminus S_0 )\cap B_\omega = (S\setminus S_0 )\cap\bigsqcup_{\gamma\in\Gamma_\omega} E_\gamma$ and $E_\gamma \subseteq B_\omega$ for all $\gamma\in\Gamma_\omega$, 
in which case we can write $B_\omega$ as the union of $B_\omega \cap S_1$ 
and $\bigsqcup_{\gamma\in\Gamma_\omega \setminus\Gamma_0} E_\gamma$, so that
\begin{align*}
|\partial_K B_\omega |
&\leq |\partial_K (B_\omega \cap S_1 )| 
+ \sum_{\gamma\in\Gamma_\omega \setminus\Gamma_0} |\partial_K  E_\gamma | \\
&\leq |K||B_\omega \cap S_1 | 
+ \frac{\delta}{2} \sum_{\gamma\in\Gamma_\omega \setminus\Gamma_0} |E_\gamma | \\
&\leq |K|\cdot \frac{\delta}{2|K|} |B_\omega | + \frac{\delta}{2} |B_\omega | \\
&= \delta |B_\omega | .
\end{align*}
This shows that $\omega\notin\Omega_0$ and hence that $\Omega_0 \subseteq\Omega_1$.
\end{proof}

\begin{lemma}\label{L-subgroup comparison}
Let $G$ be an amenable group and $H$ a subgroup of $G$. Let $G\curvearrowright X$
be a free action on a compact Hausdorff space.
Suppose that (i) the restricted action $H\curvearrowright X$ has comparison, and (ii)
for every finite set $K\subseteq G$ and $\delta > 0$ there is an open castle
$\{ S_i ,V_i \}_{i\in I}$ for the $G$-action such that each $V_i$ has diameter smaller than $\delta$, each shape $S_i$ is 
$(K,\delta )$-invariant and the upper $H$-density of $X\setminus \bigsqcup_{i\in I} S_i V_i$
is less than $\delta$. Then the action $G\curvearrowright X$ is almost finite.
\end{lemma}

\begin{proof}
If $H$ is finite then every subset of $X$ of upper $H$-density less than $|H|^{-1}$ is empty, 
so that when $\delta < |H|^{-1}$ the castles in (ii) are clopen and have footprint equal to $X$,
from which we deduce that $G\curvearrowright X$ is almost finite. 
We may thus assume that $H$ is infinite.
Let $K$ be a finite subset of $G$ and $0 < \delta < 1$. 
Choose a finite set $e \in K' \subseteq H$ with $|K'| > 1/\delta$. By assumption, there is an open castle 
$\{ (S_i ,V_i ) \}_{i\in I}$ for the $G$-action whose shapes
are $(K \cup K',\delta )$-invariant and the complement of whose footprint
has upper $H$-density at most $\delta$. 
In particular, it satisfies the first condition in the definition of almost finiteness
with respect to $K$ and $\delta$.
Choose a set $R$ of representatives for the right cosets of $H$ in $G$. 
For each $i \in I$ partition $S_i$ into subsets of right cosets of $H$, i.e.,
\[
S_i = \bigsqcup\limits_{g \in R}B_{i, g}g
\]
where each $B_{i, g}$ is contained in $H$.
Note that left translation by $K'$ preserves the right cosets of $H$. 
If $B_{i, g}$ for some $g\in R$ has cardinality less than $1/\delta$ then 
all of its elements belong to $\partial_{K'} B_{i, g}$ and so
$| \partial_{K'} B_{i, g}| \geq |B_{i, g}|$. 
Writing $L$ for the set of all $g\in R$ such that $0 < |B_{i, g}| < 1/\delta$, it follows that
\[
\sum_{g\in L} |B_{i,g} |
\leq \sum_{g\in L} | \partial_{K'} B_{i, g}|
\leq | \partial_{K'} S_i |
\leq \delta |S_i|,
\]
i.e., most elements of $S_i$ share a coset with at least $1/\delta$ other elements. 
For each $i \in I$ and $g \in R$ choose a set $B_{i, g}' \subseteq B_{i, g}$ 
with cardinality equal to $\lceil\frac{\delta}{1-\delta}|B_{i, g}|\rceil$.
Set $S_i' = \bigsqcup_{g \in R} B_{i, g}' g$ and note that when $|B_{i,g} | \geq 1/\delta$
we have $|B_{i, g}'|\leq \frac{\delta}{1-\delta}|B_{i, g}| + 1 \leq 2\delta |B_{i,g} |$,
so that
\[
|S_i' | 
\leq \sum_{g\in L} |B_{i,g} | + \sum_{g\in R\setminus L} |B_{i,g}' | 
\leq 3\delta |S_i |.
\]
Let $\mu$ be any $H$-invariant Borel probability measure on $X$. 
By construction, the set $\bigsqcup_{i\in I} S_i' V_i$ has $\mu$-measure at least 
$\frac{\delta}{1-\delta}\mu(\bigsqcup_{i\in I} S_i V_i)$, which is greater than or equal to $\delta$. 
On the other hand, since the closed set $X\setminus \bigsqcup_{i\in I} S_i V_i$ 
has upper $H$-density less than $\delta$ its $\mu$-measure is less than $\delta$,
and so our hypothesis that the $H$-action has comparison yields
\[
X\setminus \bigsqcup_{i\in I} S_i V_i \prec \bigsqcup_{i\in I} S_i^\prime V_i.
\]
Since we can take $\delta$ to be as small as we wish, this shows that 
the action $G\curvearrowright X$ is almost finite.
\end{proof}

\section{Extensions by $\Zb$}\label{S-integers}

Our goal here is to prove Theorem~\ref{T-integers}.
We will need a version of the Ornstein--Weiss quasitiling theorem \cite{OrnWei87},
which we record as Corollary~\ref{C-OW}.
The statement is a simple consequence of the following more usual version, 
which we reproduce in the form presented in \cite{KerLi16}.
A finite subset $K$ of a group $G$ is said to be {\it $\eps$-quasitiled}
by a finite collection $\sF = \{ F_1 , \dots , F_n \}$ of finite subsets of $G$ 
if there are sets $C_1 , \dots , C_n \subseteq G$ and $F_{i,c} \subseteq F_i$
with $|F_{i,c} | \geq (1-\eps )|F_i|$ for every $i=1,\dots , n$ and $c\in C_i$
such that (i) the union $\bigcup_{i=1}^n F_i C_i$ is contained in $K$ and has
cardinality at least $(1-\eps )|K|$, and (ii) the collection 
$\{ F_{i,c}c : 1\leq i\leq n, \, c\in C_i \}$ is disjoint.
As in Section~\ref{S-lemmas}, 
we say that $K$ is {\it $\sF$-tileable} if there are sets 
$C_1 , \dots , C_n \subseteq G$ such that the collection 
$\{ F_i c : 1\leq i\leq n, \, c\in C_i \}$ partitions $K$.

\begin{theorem}\label{T-OW} 
Let $G$ be a group.
Let $0 < \eps < \frac12$ and let $m\in\Nb$ be such that $(1 - \eps /2)^m < \eps$.
Let $e\in F_1 \subseteq F_2 \subseteq\cdots\subseteq F_m$ be finite subsets
of $G$ such that for each $k=2,\dots ,m$ the set $F_k$ is $(F_{k-1} , \eps /8)$-invariant.
Then every $(F_m ,\eps /4)$-invariant finite subset of $G$ is $\eps$-quasitiled
by $\{ F_1 , \dots , F_m \}$.
\end{theorem}

\begin{corollary}\label{C-OW}
Let $G$ be an amenable group. Let $0 < \eps < \frac12$. 
Let $K$ be a finite subset of $G$ and $\delta > 0$.
Then there exists a finite collection $\sF$ of $(K,\delta )$-invariant finite subsets of $G$
containing $e$ such that for every $(\bigcup\sF ,\eps /4)$-invariant
finite set $E\subseteq G$ there is an $\sF$-tileable $E' \subseteq E$ 
satisfying $|E' | \geq (1-\eps )|E|$. 
\end{corollary}

\begin{proof}
Let $m\in\Nb$ be such that $(1 - \eps /2)^m < \eps$.
Since $G$ is amenable we can find $e\in F_1 \subseteq F_2 \subseteq\cdots\subseteq F_m$
as in the statement of Theorem~\ref{T-OW}. Write $\sF$ for the (finite) collection of all 
sets $F$ such that for some $j=1,\dots ,m$ we have $F\subseteq F_j$ and $|F| \geq (1-\eps )|F_j|$.
In view of the definition of $\eps$-quasitiling, Theorem~\ref{T-OW} 
then tells us that for every $(F_m ,\eps /4)$-invariant finite set $E\subseteq G$ 
there is an $\sF$-tileable $E' \subseteq E$ such that $|E' | \geq (1-\eps )|E|$.
As $F_m = \bigcup\sF$ this yields the conclusion.
\end{proof}

Let $H$ be a countable group and $\alpha$ an automorphism of $H$,
and form the corresponding semidirect product $H\rtimes\Zb$.
Inside $H\rtimes\Zb$ we view $\Zb$ multiplicatively as the group $\langle g \rangle$ 
with generator $g$ satisfying $g^m tg^{-m} = \alpha^m (t)$ for all 
$m\in\Zb$ and $t\in H$.
When we say an {\it interval} in $\langle g \rangle$ 
we mean a set of the form $\{ g^m , g^{m-1} , \dots , g^n \}$ 
for some integers $m\leq n$, and by the {\it length} of this interval we mean $n-m$.

\begin{lemma}\label{L-density}
Let $H\rtimes\Zb \curvearrowright X$ be a continuous action on 
a compact metrizable space. 
Then $\overline{D}_H (gA) = \overline{D}_H (A)$ for all $A\subseteq X$. 
\end{lemma}

\begin{proof}
For every $A\subseteq X$ we have, with $F$ ranging over the nonempty finite subsets of $H$,
\begin{align*}
\overline{D}_H (A) 
= \inf_F \sup_{x\in X} \frac{1}{|F|}\sum_{s\in F} 1_A (sx) 
&= \inf_F \sup_{x\in X} \frac{1}{|F|}\sum_{s\in F} 1_{gA} (gsx) \\
&= \inf_F \sup_{x\in X} \frac{1}{|\alpha (F)|}\sum_{s\in F} 1_{gA} (\alpha (s)gx) \\
&= \inf_F \sup_{x\in X} \frac{1}{|\alpha (F)|}\sum_{t\in \alpha (F)} 1_{gA} (tgx) \\
&= \overline{D}_H (gA) . \qedhere
\end{align*}
\end{proof}

\begin{theorem}\label{T-integers}
Let $H\rtimes\Zb \curvearrowright X$ be a free continuous action on a compact metric space.
Suppose that the restricted action $H\curvearrowright X$ is almost finite. 
Then the action $H\rtimes\Zb \curvearrowright X$ is almost finite.
\end{theorem}

\begin{proof}
Let $0 < \eps < \frac16$.
Let $K$ be a finite subset of $H$.
Take an $r\in\Nb$ large enough so that any interval in $\langle g \rangle$ of length
at least $\eps^5 r$ is $(\{ g \},\eps )$-invariant.
Denote by $A$ the symmetric interval $\{ g^{-r} ,\dots ,g^r \}$.

Set $K' = \bigcup_{m=-4r}^{4r} \alpha^m (K)$ and 
$K'' = \bigcup_{m=-4r}^{4r} \alpha^m (K') = \bigcup_{m=-8r}^{8r} \alpha^m (K)$.
By Corollary~\ref{C-OW} there is a finite collection $\sF$ of 
$(K'',\eps^{15} /(8|K'|^2 (2r+1)r^{15}))$-invariant finite subsets 
of $H$ containing $e$ such that for every $(\bigcup\sF ,\eps /(8r))$-invariant
finite set $F\subseteq H$ there exists an $\sF$-tileable $F' \subseteq F$ satisfying 
$|F' | \geq (1-\frac{\eps}{2r} )|F|.$

Since $H\curvearrowright X$ is almost finite it is almost finite in measure,
and so we can find an open castle $\{ (S_k , V_k ) \}_{k=1}^n$
for this action whose shapes are $(((\bigcup\sF )(\bigcup\sF )^{-1} )^2,\eps^{10} /(4|K'|r^{10})))$-invariant and 
whose footprint $\bigsqcup_{k=1}^n S_k V_k$ has lower $H$-density at least $1-\frac{\eps}{2r}$.
The proof of Theorem~5.6 in \cite{KerSza20} shows that we may
assume the boundary of each level of each tower in the castle
to have zero upper $H$-density.
By Theorem~5.5 of \cite{KerSza20}, for any $k\in \{ 1,\dots , n\}$ we can find a finite
disjoint collection $\sU$ of open subsets of $V_k$ whose diameters
are as small as we wish such that the set $V_k \setminus \bigcup\sU$
has zero upper $H$-density.
Since the action $H\rtimes\Zb \curvearrowright X$ is free, 
we may therefore furthermore assume, 
by replacing each tower $(S_k ,V_k )$ with a collection of towers 
with shape $S_k$ whose bases are the members of a suitable collection 
of open subsets of $V_k$ of the type just described,
that $(A^2 S_k , V_k )$ is a tower for every $k=1,\dots ,n$.
For each $k=1,\dots ,n$, since the shape $S_k$ is 
$(\bigcup\sF ,\eps /(8r))$-invariant we can find, by the previous paragraph, 
an $\sF$-tileable $B_k \subseteq S_k$ satisfying $|B_k | \geq (1-\frac{\eps}{2r} )|S_k |$.

Set $\sF' = \{ \alpha^m (F) : -4r\leq m\leq 4r, \, F\in\sF \}$. 
Since each member of $\sF$ is $(K'',\eps^{15}/(8|K'|(8r+1)r^{15}))$-invariant,
each member of $\sF'$ is $(K',\eps^{15}/(8|K'|(8r+1)r^{15}))$-invariant, as is easily checked
using the boundary commutation formula 
$\partial_{K'} \alpha^i (F) = \alpha^m (\partial_{\alpha^{-m} (K')} F)$ for $m\in\Zb$ and finite sets $F\subseteq H$.
Moreover, the fact that each $B_k$ is 
$\sF$-tileable implies that for every $k=1,\dots ,n$ and $m=-4r,\dots ,4r$ the set 
$\alpha^m (B_k )$ is $\sF'$-tileable.

We now carry out a recursive disjointification process over $k$. This will be similar
to the Ornstein--Weiss-type argument in \cite{ConJacKerMarSewTuc18,KerSza20} 
for producing dynamical tilings, but with two significant 
differences: (i) we carry out the disjointification procedure in one single recursion without having to repeat it across different scales,
and (ii) the decision to retain a new piece at a given stage depends on whether its proportion 
in the ambient tile is greater than $\eps$ (actually $\eps^4$ in our case)
instead of greater than $1-\eps$ (the latter is what makes the recursion over different scales in the Ornstein--Weiss setting necessary, since at every scale only a small part of the space gets covered).

Write $\sT_k$ for the collection of all subsets of $AB_k$.
Write $\sT_k^+$ for the collection of all $T\in\sT_k$ such that $|T|\geq \eps^4 |AB_k |$.
We will recursively construct sets $Z_1 , \dots , Z_n \subseteq X$ 
and (not necessarily open) castles $\{ (T,V_{k,T} ) \}_{T\in\sT_k^+}$ for $k=1,\dots , n$
such that $Z_k = Z_{k-1} \sqcup \bigsqcup_{T\in\sT_k^+} TV_{k,T}$ for $k=2,\dots , n$. Many
of the sets $V_{k,T}$ will be empty, in part because nonemptiness forces some extra structure on the corresponding $T$
(as explained below), but we will worry about this after the construction so as to not complicate notation.

We begin by setting $V_{1,T} = V_1$ for $T=AB_1$ and $V_{1,T} = \emptyset$ for $T\neq AB_1$.
We also put $Z_1 = AB_1 V_1$.

Suppose now that $1 < k \le n$ and that we have constructed $Z_1 , \dots , Z_{k-1}$ and the castles
$\{ (T,V_{j,T} ) \}_{T\in\sT_j^+}$ for $j=1,\dots , k-1$. For each $T\in\sT_k$ set
\begin{gather}\label{E-int-union}
V_{k,T} = V_k \cap \bigg(\bigcup_{s\in AB_k \setminus T} s^{-1} Z_{k-1} \bigg) 
\cap \bigg(\bigcup_{s\in T} (X\setminus s^{-1} Z_{k-1} ) \bigg) .
\end{gather}
Then $\{ (T,V_{k,T} ) \}_{T\in\sT_k^+}$ is a castle and we define $Z_k$ to be the union of the two sets $Z_{k-1}$ 
and $\bigsqcup_{T\in\sT_k^+} TV_{k,T}$, which are disjoint (in this construction we only care about
$V_{k,T}$ for $T\in\sT_k^+$, but for later use we have defined $V_{k,T}$ above for all $T\in\sT_k$).
This completes the recursion. A very important observation is that, setting $V_k^+ = \bigsqcup_{T\in\sT_k^+} V_{k,T} \subseteq V_k$, we have
\begin{align}\label{eq: Z_k is union}
Z_k = \bigcup_{j=1}^k AB_jV_j^+.
\end{align}

Write $Q$ for the remainder $X\setminus \bigsqcup_{k=1}^n B_kV_k$ of the castle $\{ (B_k ,V_k )\}_{k=1}^n$.
It will be convenient to have the following dual picture of the 
tower partition $Q\sqcup \bigsqcup_{k=1}^nB_kV_k$ in terms of the partial orbits within each tower.
We define the equivalence relation $E$ on $X$ as the smallest under which $xEgx$ 
whenever $x$ is in some base $V_k$ and $g$ is in the corresponding shape $B_k$.
Each equivalence class $[x]_E$ either belongs to a unique tower $B_kV_k$ or is a singleton in the remainder $Q$.
For $m\in\Zb$ we have
 \[
X = g^m \bigg(Q\sqcup \bigsqcup_{k=1}^nB_kV_k \bigg) =  g^m Q \sqcup \bigsqcup_{k=1}^n \left(g^mB_kg^{-m}\right)g^mV_k, 
\]
that is, $X$ is also partitioned by the castle $\{(g^mB_kg^{-m}, g^mV_k )\}_{k=1}^n$ together with its remainder $g^mQ$. 
We denote the corresponding equivalence relation by $(g^m)_*E$.
Note that $[x]_{(g^m)_*E} = g^m\left([x]_E\right)$ and therefore $x(g^m)_*Ey$ if and only if $g^{-m}xEg^{-m}y$. 
For $q\in\Nb$ define
\begin{gather*}
E^q = \bigcap_{m=-q}^q (g^m)_*E.
\end{gather*}
Observe that for any $1 \leq k \leq n$ it follows from \eqref{eq: Z_k is union} that if $xE^ry$ then
\[
x \in Z_k \Leftrightarrow y \in Z_k.
\]

Now let $x \in V_k$ for some $1\leq k\leq n$ and consider the set $AB_kx \setminus Z_{k-1} = AB_kx \setminus \bigcup_{j=1}^{k-1}AB_jV_j^+$.
For a fixed element $h \in B_k$ the set $Ahx \setminus Z_{k-1}$ has the form $A_h hx$ 
where $A_h$ is a subinterval of $A$. Moreover, if $h^\prime x \in [hx]_{E^{2r}}$ then $g^mh^\prime x \in [g^mhx]_{E^{r}}$ 
for every $m=-r, \ldots, r$ and therefore $A_{h^\prime} = A_h$. 
In other words, the interval $A_h$ only depends on the $E^{2r}$-equivalence class of $hx$. 

Consider now an equivalence relation $E^\prime$ defined by setting $h_0xE^\prime h_1x$ if $A_{h_0} = A_{h_1}$. By the previous observation, every $E^\prime$-equivalence class is a disjoint union of $E^{2r}$-equivalence classes. The 
set of points in $AB_kx$ that are not covered by $Z_{k-1}$, 
which is equal to $Tx$ where $T\in\sT_k$ is such that $x\in V_{k,T}$,
can be partitioned into ``rectangles'' of the form $A_B Bx$ where $Bx$ is an $E^\prime$-equivalence class
and $A_B$ is equal to $A_h$ for any $h$ with $h\in B$.
By identifying $T$ with $Tx$ via $t\mapsto tx$ this yields a partition of $T$ itself into sets of the form $A_B B$, and this
partition is the same for all $x$ in $V_{k,T}$.

Write $\sN_k$ for the collection of all $T\in\sT_k$ for which $V_{k,T}$ is nonempty.
Let $T\in\sN_k$.
Write $\sB_{k,T}$ for the collection of all sets $B$ appearing in the common partition of $T$ into rectangles 
of the form $A_B B$ for points in $V_{k,T}$, as described above. We can thus express $T$ as $\bigsqcup_{B\in\sB_{k,T}} A_B B$.
Write $\sC_{k,T}$ for the collection of all $B\in\sB_{k,T}$ such that $B$ is $(K',\eps )$-invariant.
Set $C_{k,T} = \bigsqcup \sC_{k,T}$.
Since $S_k$ is $(((\bigcup\sF )(\bigcup\sF )^{-1} )^2,\eps^{10} /(4|K'|r^{10}))$-invariant, the members of
$\sF'$ are $(K',\eps^{15} /(8|K'|(8r+1)r^{15}))$-invariant, and the sets $\alpha^m (B_k )$ for $m=-4r,\dots ,4r$ and $k=1,\dots , n$
are $\sF'$-tileable, we can apply Lemma~\ref{L-intersections} (taking there $n=8r+1$, $\delta = (\eps/r)^5$, $K=K'$
and the sets $B_{i,j}$ to be the right translates of $\alpha^i (B_k )$ which are at play in the description
of an intersection pattern within an orbit) to see that 
\begin{align}\label{E-invt}
|C_{k,T} | \geq (1-(\eps/r)^5 ) |B_k | > (1-\eps^5 ) |B_k | .
\end{align}

Set $\sN_k^+ = \sT_k^+ \cap \sN_k$ and $\sN_k^- = \sN_k \setminus \sT_k^+$.

Given a $T\in\sN_k^+$, let us verify that it is $(K\cup \{ g \} ,2\eps )$-invariant.
Define $F_{k,T,1} = \bigsqcup_{h\in C_{k,T}} A_h h$ 
where $A_h$ is such that $A_h h V_{k,T} = Ah V_{k,T} \setminus Z_{k-1}$ (in accordance
with our orbitwise discussion above), and set $F_{k,T,0} = T \setminus F_{k,T,1}$.
By the definition of $C_{k,T}$ and the boundary commutation formula
$\partial_K (g^m F) = g^{-m} \partial_{\alpha^{-m} (K)} F$ for $m\in\Zb$ and finite sets $F\subseteq H$, 
the set $F_{k,T,1}$ is $(K,\eps )$-invariant, and since
\[
|F_{k,T,0} | \leq |A||B_k\setminus C_{k,T} | 
\stackrel{(\ref{E-invt})}{\leq} \eps^5 |A||B_k| \leq \eps |T|
\]
we thus see that the set $T = F_{k,T,0} \sqcup F_{k,T,1}$ is $(K,2\eps )$-invariant.
On the other hand, defining $D_{k,T,0}$ to be the set of all $h\in B_k$ such that $|A_h | \leq \eps^5 r$ 
and setting $D_{k,T,1} = B_k \setminus D_{k,T,0}$, we have
\[
\bigg|\bigsqcup_{h\in D_{k,T,0}} A_h h \bigg| \leq \eps^5 r |B_k | = \eps^5 \cdot \frac{r}{2r+1} |AB_k | 
\leq \eps^5\cdot \frac{1}{3\eps^4}|T| < \eps |T|
\]
while $\bigsqcup_{h\in D_{k,T,1}} A_h h$ is $(\{ g \} , \eps )$-invariant by our choice of $r$, so that 
$T = \bigsqcup_{h\in D_{k,T,0}} A_h h \sqcup \bigsqcup_{h\in D_{k,T,1}} A_h h$ is $(\{ g \} , 2\eps )$-invariant.

Now let $T\in\sN_k^-$ and define $B_{k,T,0}$ to be the set of all $h\in B_k$ such that $|A_h | \leq \eps r$ 
(with $A_h$ having the same meaning as above with respect to $k$ and $T$) and set $B_{k,T,1} = B_k \setminus B_{k,T,0}$.
Then 
\begin{align*}
|B_{k,T,1} | \cdot \eps r \leq |T| < \eps^4 |AB_k | = \eps^4 (2r+1)|B_k |
\end{align*}
and hence
\begin{align}\label{E-density}
|B_{k,T,1} | < 3\eps^3 |B_k |,
\end{align}
a fact that we will use later.
Recall that $Z_n = \bigsqcup_{k=1}^n \bigsqcup_{T\in\sN_k^+} TV_{k,T}$ (the footprint of the castle we have recursively constructed)
and set
\[
W = \bigsqcup_{k=1}^n B_kV_k \setminus Z_n.
\]

We will next carry out a process by which we donate some of the points of $W$ to the towers in our recursively constructed castle.
By splitting up the towers according to how their shapes have been amplified orbit by orbit through this donation process, 
we obtain a new castle with a much larger collection of towers. The donation process will be carried out so that
the amplified shapes are proportionally not much larger and hence still approximately invariant,
and it will be sufficiently algorithmic so that the boundaries of the levels of the new towers will still have zero upper $H$-density.
Once we have done this, it will remain to shave off the boundaries of the tower levels  
so that the towers become open, which will involve
discarding a set of zero upper $H$-density, and then finally check that the remainder of the resulting castle
has small $H$-density, for then we can apply Lemma~\ref{L-subgroup comparison} to finish.

To set up the donation operation, suppose that we have a subset of $W$ of the form $Jx$ for some $x \in X$ and some 
(possibly infinite) interval $J$
in $\langle g \rangle$ which is maximal with respect to the inclusion of $Jx$ in $W$. We will either donate
none of $Jx$ or all of it.

If $|J| > \eps r$ we do not donate $Jx$. 
Assume then that $|J| \le \eps r$. By replacing $x$ if necessary, we may assume that $J$
has the form $\{ g^0 , g^{1} , \dots , g^q \}$. Consider the point $g^{-1}x$. Since it is not in $W$, it is either in $Q = X \setminus \bigsqcup_{k=1}^n B_kV_k$ or in $Z_n$. In the first case, we do not donate $Jx$ and only note that $Jx \subseteq \{g^1, \ldots g^{\lceil \eps r \rceil + 1}\}Q$. Suppose now that $g^{-1}x \in Z_n$, that is, $g^{-1}x = ty$ for some $y \in V_{k, T}$ and $t \in T \in \sN_k^+$. Consider the equivalence class $[ty]_{E^{2r}}$. If it is $(K^\prime, \eps)$-invariant 
then we donate $Jx$ to the tower containing $Ty$, thereby enlarging the partial orbit $Ty$ within this tower.
Otherwise, we again do not donate $Jx$.

Note that the sets $W, Q, TV_{k, T}$, and $Z_n$ are $E^r$-invariant, as can be easily seen from \eqref{eq: Z_k is union}. Hence, if the interval $Jx$ was donated to a tower via the partial orbit $Ty$, then, in fact, the whole set $Jg[g^{-1}x]_{E^{2r}}$ was donated to this tower via $Ty$.
Through this procedure we obtain a new collection of towers $(T^\sharp , V_{k, T^\sharp} )$. These clearly form a castle, which we denote by $\sL_0$.
By our definition of $K'$ and the boundary commutation formula 
$\partial_K (g^m F) = g^{-m} \partial_{\alpha^{-m} (K)} F$ for all 
$m\in\Zb$ and finite sets $F\subseteq H$, we see that each new shape $T^\sharp$ is still $(K\cup \{ g \} , \eps )$-invariant
(the invariance in the $\Zb$-direction only improves since the $\langle g \rangle$ cross sections 
can only become longer intervals). Given that the set $H \cup \{ g \}$ generates $H\rtimes\Zb$,
a standard exercise then shows that we can make the shape as left invariant as we wish
by making an appropriate choice of $K$ and taking $\eps$ small enough.

By Lemma~\ref{L-density}, the collection 
of subsets of $X$ with upper $H$-density zero
is $H\rtimes\Zb$-invariant, and so the sets $g(V_k \setminus V_k^\circ )$ 
for $k=1,\dots ,n$ and $g \in H\rtimes\Zb$ all belong to this collection given that each $V_k \setminus V_k^\circ$ does.
Since this collection is also an algebra
we therefore deduce, in view of the algorithmic way in which the above construction proceeded 
based on intersection patterns, that the boundaries of the levels of the tower in $\sL_0$ 
all have zero $H$-density. In view of the uniform continuity of the individual homeomorphisms making up the action
of $H\rtimes\Zb$, 
we can also make the levels of the towers in $\sL_0$ have as small a diameter as we wish by taking the diameters of 
the bases $V_1 , \dots , V_n$ of the initial towers to be sufficiently small.

Write $\sL$ for the open castle consisting of the towers $(T^\sharp , V_{k,T^\sharp}^\circ )$.
This castle will be our witness for the almost finiteness of the $H\rtimes\Zb$-action.
It remains to check that its remainder satisfies the smallness condition in the definition of almost finiteness,
and for this it suffices to show that it is small in upper $H$-density, since the restriction action $H\curvearrowright X$
has comparison and we can apply Lemma~\ref{L-subgroup comparison}. 

More precisely, we will verify that the remainder of the castle $\sL$ has upper $H$-density at most $7\eps$. 
By construction, it is contained in the union of the sets
\begin{enumerate}
\item $g(V_{k} \setminus V_{k}^\circ)$ for all $k$ and $g \in H\rtimes\Zb$,

\item $\{g^1, \ldots g^{\lceil \eps r \rceil + 1}\}Q$,

\item the union of all intervals of the form $Jx$ in $W$ with $|J| > \eps r$, denoted $W^\prime$, and

\item the union of sets of the form $\{g^1, \ldots g^{\lceil \eps r \rceil + 1}\}[x]_{E^{2r}}$ for all $[x]_{E^{2r}}$ that are not $(K^\prime, \eps)$-invariant. 
\end{enumerate}
We already observed that the sets in (i) have zero upper $H$-density. For (ii), recall that $X \setminus \bigsqcup_{k=1}^n S_kV_k$ has upper $H$-density less than $1 - \frac{\eps}{2r}$ and that $|B_k| > (1 - \frac{\eps}{2r})|S_k|$. It follows that $\overline{D}_H(Q) < \eps/r$ and therefore
\[
\overline{D}_H(\{g^1, \ldots g^{\lceil \eps r \rceil + 1}\}Q) \le \lceil \eps r \rceil\overline{D}_H(Q) < \eps.
\]

To estimate (iii), consider a point $x \in V_{k, T}$ for some $T \in \sN_k$ and let $hx \in B_kx \cap W^\prime$. As $hx$ is not in $Z_n$, we conclude that $T \in \sN_k^-$. Moreover, it is clear that $|A_h| > \eps r$ and thus $h \in B_{k, T, 1}$. Applying \eqref{E-density} then yields
\[
|W^\prime \cap B_kx| < 3\eps^3|B_k|
\]
for any $k=1, 2, \ldots, n$ and any $x \in V_k$. Thus,
\[
\overline{D}_H(W^\prime) \le 3\eps^3\overline{D}_H\bigg(\bigsqcup_{k=1}^nB_kV_k\bigg) + \overline{D}_H(Q) < \eps.
\]

Finally, to estimate (iv), let $x \in X \setminus \bigcup_{m=-2r}^{2r}g^mQ$ and suppose that $[x]_{E^{2r}}$ is not $(K', \eps)$-invariant. Clearly, $x = hy$ for some $y \in V_k$ and $h \in B_k$ and Lemma 4.1 shows (see \eqref{E-invt} for a similar calculation) that there are at most $(\eps/r)^5|B_k|$ such points $x$ in the set $B_ky$. Thus, the set in item (iv) has upper $H$-density at most
\[
(4r + \lceil \eps r \rceil + 1)\overline{D}_H(Q) + \lceil \eps r \rceil (\eps/r)^5 < 5\eps.
\]

Since upper $H$-density is monotone with respect to inclusions and subadditive with respect to
finite unions, we conclude from the above estimates that the remainder of the castle $\sL$ 
has upper $H$-density at most $7\eps$.
By hypothesis the restricted action $H\curvearrowright X$ is almost finite
and hence has comparison, and so it follows by Lemma~\ref{L-subgroup comparison} that the action
$H\rtimes\Zb \curvearrowright X$ is almost finite, as desired.
\end{proof}

\end{document}